\documentclass[table]{amsart}
\usepackage[margin=1.1in]{geometry}
\usepackage{amscd,amsmath,amsxtra,amsthm,amssymb,stmaryrd,xr,mathrsfs,mathtools,enumerate,commath, comment, mathtools}
\usepackage{stmaryrd}
\usepackage{multirow}
\usepackage{xcolor}
\usepackage{commath}
\usepackage{comment}
\usepackage{tikz-cd}
\usepackage{longtable} 
\usepackage{pdflscape} 
\usepackage{booktabs}
\usepackage{hyperref}
\definecolor{vegasgold}{rgb}{0.77, 0.7, 0.35}
\definecolor{darkgoldenrod}{rgb}{0.72, 0.53, 0.04}
\definecolor{gold(metallic)}{rgb}{0.83, 0.69, 0.22}
\hypersetup{
 colorlinks=true,
 linkcolor=darkgoldenrod,
 filecolor=brown,      
 urlcolor=gold(metallic),
 citecolor=darkgoldenrod,
 }
\newtheorem{lthm}{Theorem}

\usepackage[all,cmtip]{xy}

\DeclareFontFamily{U}{wncy}{}
\DeclareFontShape{U}{wncy}{m}{n}{<->wncyr10}{}
\DeclareSymbolFont{mcy}{U}{wncy}{m}{n}
\DeclareMathSymbol{\Sh}{\mathord}{mcy}{"58}
\usepackage[T2A,T1]{fontenc}
\usepackage[OT2,T1]{fontenc}

\newtheorem{theorem}{Theorem}[section]
\newtheorem{lemma}[theorem]{Lemma}
\newtheorem{ass}[theorem]{Assumption}
\newtheorem*{theorem*}{Theorem}
\newtheorem*{ass*}{Assumption}
\newtheorem{definition}[theorem]{Definition}

\newtheorem{remark}[theorem]{Remark}

\newtheorem{proposition}[theorem]{Proposition}

\newcommand{\cK}{\mathcal{K}}

\newcommand{\Z}{\mathbb{Z}}
\newcommand{\Q}{\mathbb{Q}}
\newcommand{\F}{\mathbb{F}}

\newcommand{\cL}{\mathcal{L}}

\newcommand{\cO}{\mathcal{O}}
\newcommand{\cS}{\mathcal{S}}

\newcommand{\Sel}{\mathrm{Sel}}

\newcommand{\Gr}{\mathrm{Gr}}

\newcommand{\op}[1]{\operatorname{#1}}

\newcommand\mtx[4] { \left( {\begin{array}{cc}
 #1 & #2 \\
 #3 & #4 \\
 \end{array} } \right)}

 \DeclareMathSymbol{\Sh}{\mathord}{mcy}{"58}
 \makeatletter
\newcommand{\mylabel}[2]{#2\def\@currentlabel{#2}\label{#1}}
\makeatother

\numberwithin{equation}{section}

\begin{document}

\title[Selmer stability of congruent Galois representations]{Selmer stability in families of congruent Galois representations}

\author[A.~Ray]{Anwesh Ray}
\address[Ray]{Chennai Mathematical Institute, H1, SIPCOT IT Park, Kelambakkam, Siruseri, Tamil Nadu 603103, India}
\email{anwesh@cmi.ac.in}


\maketitle

\begin{abstract}
In this article, we study the variation of Selmer groups in families of modular Galois representations that are congruent modulo a fixed prime $p \geq 5$. Motivated by analogies with Goldfeld’s conjecture on ranks in quadratic twist families of elliptic curves, we investigate the stability of Selmer groups defined over $\Q$ via Greenberg's local conditions under congruences of residual Galois representations. Let $X$ be a positive real number. Fix a residual representation $\bar{\rho}$ and a corresponding modular form $f$ of weight $2$ and optimal level. We count the number of level-raising modular forms $g$ of weight $2$ that are congruent to $f$ modulo $p$, with level $N_g\leq X$, such that the $p$-rank of the Selmer groups of $g$ equals that of $f$. Under some mild assumptions on $\bar{\rho}$, we prove that this count grows at least as fast as $X (\log X)^{\alpha - 1}$ as $X \to \infty$, for an explicit constant $\alpha > 0$. The main result is a partial generalization of theorems of Ono and Skinner on rank-zero quadratic twists to the setting of modular forms and Selmer groups.
\end{abstract}

\section{Introduction}

\subsection{Motivation and background}
\par The questions studied in this article are motivated by a conjecture of Goldfeld \cite{Goldfeld}, which concerns the distribution of Mordell–Weil ranks in families of quadratic twists of a fixed elliptic curve over $\mathbb{Q}$. Given an elliptic curve $E/\mathbb{Q}$, one can consider its quadratic twists $E_d$, indexed by square-free integers $d$. Goldfeld conjectured that, asymptotically, 50 percent of these twists should have rank 0 and 50 percent should have rank 1, with 0 percent having rank greater than 1. This conjecture is informed by analogies with random matrix theory, the Birch and Swinnerton-Dyer conjecture, and probabilistic models for Selmer groups. An important feature of the quadratic twist family is that the mod-2 Galois representation remains fixed: for all $d$, one has an isomorphism $E_d[2] \simeq E[2]$ of Galois modules. This \emph{mod-2 congruence} implies that the 2-Selmer groups of $E_d$ and $E$ are closely related—a theme that plays a central role in the work of Klagsbrun–Mazur–Rubin \cite{KMR}, Kriz–Li \cite{KrizLi}, and Smith \cite{smith2022distribution1, smith2022distribution2}. It is natural to study generalizations of this theme to $p$ congruences between modular Galois representations. We shall work with primes $p\geq 5$. Level--raising results of Diamond--Taylor \cite{diamond1994non} provide us with parameterizations of families of $p$-congruent Galois representations. Congruences have also played a significant role in Iwasawa theory. For a prime $p$, two elliptic curves $E$ and $E'$ over $\mathbb{Q}$ are said to be \emph{$p$-congruent} if there exists a Galois-module isomorphism $E[p] \simeq E'[p]$. Greenberg and Vatsal \cite{greenberg2000iwasawa} showed that for such $p$-congruent curves, the Iwasawa invariants of their $p$-primary Selmer groups over the cyclotomic $\mathbb{Z}_p$-extension of $\mathbb{Q}$ are explicitly related.

\subsection{Main result}
\par In this article, we study the behavior of Selmer groups in $p$-congruent families of modular Galois representations, with particular emphasis on their stability under congruences. We consider Selmer groups defined over $\Q$ by Greenberg's local conditions and investigate analogues of phenomena observed in the context of elliptic curves, aiming to understand how congruences between Galois representations influence the global structure of Selmer groups. The Selmer groups considered are closely related to Bloch--Kato Selmer groups, however are not the same.

\par We fix a prime $p \geq 5$, a $p$-adic field $\cK$ with valuation ring $\cO$, uniformizer $\varpi$, and residue field $\kappa$. Denote by $\chi$, the $p$-adic cyclotomic character. Let
$\bar{\rho} : \op{G}_\Q \to \op{GL}_2(\kappa)$
be a surjective, odd Galois representation with determinant $\bar{\chi}$, the mod-$p$ reduction of $\chi$. Let $N_{\bar{\rho}}$ be the prime to $p$ part of the Artin conductor of $\bar{\rho}$. By results of Khare and Wintenberger \cite{khare2009serre, khare2009serre2}, there exists a weight 2 Hecke eigenform $f$ of trivial nebentypus and optimal level $N_f = N_{\bar{\rho}}$ such that $\bar{\rho}_f = \bar{\rho}$. Associated with $f$ is a modular Galois representation $\rho_f:\op{G}_\Q\rightarrow \op{Aut}_{\cK} (V_f)\xrightarrow{\sim}\op{GL}_2(\cK)$, where $V_f\simeq \cK^2$ is the underlying vector space. There exists a Galois stable lattice $T_f\subset V_f$. We set $A_f:=V_f/T_f\simeq (\cK/\cO)^2$. Denote by $\op{Sel}_{\op{Gr}}(A_f/\Q)$ the Selmer group of $f$ over $\Q$ defined by Greenberg local conditions. This is a cofinitely generated $\varpi$-primary $\cO$-module, which contains the $\varpi$-primary Bloch--Kato Selmer group as a finite index submodule. For \( X > 0 \), let \( \cS_f(X) \) denote the set of natural numbers \( N \leq X \) such that \( N = N_g \) for some Hecke eigencuspform \( g \) of weight \( 2 \) satisfying:
\begin{itemize}
    \item \( \bar{\rho}_g \simeq \bar{\rho}_f \),
    \item $p\nmid N_g$,
    \item \( \dim \op{Sel}_{\op{Gr}}(A_g/\Q)[\varpi] = \dim \op{Sel}_{\op{Gr}}(A_f/\Q)[\varpi] \).
\end{itemize}
Define \( N_f(X) := \# \cS_f(X) \).

\begin{lthm}\label{main thm of intro}
With respect to notation as above, assume that no prime $\ell\equiv \pm 1\pmod{p}$ divides $N_{\bar{\rho}}$. Then,
$$
N_f(X) \gg X (\log X)^{\alpha - 1},
$$
\noindent where $\alpha := \frac{p - 3}{(p - 1)^2}$.
\end{lthm}

This theorem may be viewed as an analogue, in the setting of modular forms and Galois representations, of classical results related to Goldfeld’s conjecture. For an elliptic curve $E/\Q$, let $N_E^0(X)$ denote the number of squarefree integers $d$ with $|d| < X$ for which the quadratic twist $E^{(d)}$ has analytic rank 0. Ono and Skinner \cite{OnoSkinner} showed that $N_E^0(X) \gg \frac{X}{\log X}$. Later, Ono \cite{OnoCrelle} refined this result: if $E(\Q)[2] = 0$, then there exists an explicit constant $\alpha(E) \in (0,1)$ such that

$$
N_E^0(X) \gg \frac{X}{(\log X)^{1 - \alpha(E)}}.
$$
\noindent For $2$-Selmer groups of elliptic curves in quadratic twist families of elliptic curves over number fields, results similar to the Theorem above have been proven by Mazur and Rubin (see \cite[Theorem 1.4]{MRH10}).

\section*{Statements and Declarations}
\subsection*{Acknowledgement}
We thank the anonymous referee for the excellent report.
\subsection*{Conflict of interest} The author reports that there are no conflicts of interest to declare.

\subsection*{Data Availability} There is no data associated to the results of this manuscript.

\section{Modular forms and their Selmer groups}

\par In this section, we discuss modular Galois representations, congruences and associated Selmer groups. We begin by setting relevant notation.
\subsection{Notation}
\begin{itemize}
    \item Throughout, \( p \) will be an odd prime number.  
    \item Let $\F_p$ be the finite field with $p$ elements. Given an $\F_p$ vector space $V$, set $V^\vee:=\op{Hom}\left(V, \F_p\right)$.
    \item Let \( \op{G}_\Q := \op{Gal}(\bar{\Q}/\Q) \) be the absolute Galois group of \( \Q \).
    \item Denote by $\chi:\op{G}_{\Q}\rightarrow \Z_p^\times$ the $p$-adic cyclotomic character and let $\bar{\chi}$ be its mod-$p$ reduction.
    \item Let \( \tau \) be a variable in the complex upper half-plane. Setting \( q := \exp(2\pi i \tau) \), let \( f = \sum_{n=1}^\infty a_n(f) q^n \) be a normalized Hecke eigenform of weight \( k \geq 2\) and level \( N_f \). Let \( F_f \subset \bar{\Q}\) denote the number field generated by its Fourier coefficients.
    \item For each prime \( \ell \), let \( \bar{\Q}_\ell \) be a fixed algebraic closure of \( \Q_\ell \), and fix an embedding $\iota_\ell: \bar{\Q} \hookrightarrow \bar{\Q}_\ell$. Let \( \op{G}_\ell := \op{Gal}(\bar{\Q}_\ell/\Q_\ell) \) be the absolute Galois group of \( \Q_\ell \). We adopt the short hand notation $H^i(\Q_\ell, \cdot):=H^i(\op{G}_\ell, \cdot)$. The embedding $\iota_\ell$ induces an inclusion 
    \[\iota_\ell^*: \op{G}_\ell\hookrightarrow \op{G}_\Q.\]
    \item The inertia subgroup of $\op{G}_\ell$ is denoted by $\op{I}_\ell$ and let $\sigma_\ell\in \op{G}_\ell/\op{I}_\ell$ be the Frobenius element. When there is no cause for confusion, we occasionally choose a lift in $\op{G}_\ell$ of the Frobenius, which we shall also denote by $\sigma_\ell$.
    \item Given a module $M$ over $\op{G}_\ell$, the unramified classes in $H^1(\Q_\ell, M)$ are defined as follows:
\[H^1_{\op{nr}}(\Q_\ell, M):=\op{image}\left\{ H^1(\op{G}_\ell/\op{I}_\ell, M^{\op{I}_\ell})\longrightarrow H^1(\Q_\ell, M)\right\}.\]
    \item The embedding \( \iota_\ell \) induces an injection \( \iota_\ell^* : \op{G}_\ell \hookrightarrow \op{G}_\Q \). Let $\cK/\Q_p$ be the finite extension generated by $\iota_p(F_f)$, let $\cO$ be the valuation ring of $\cK$ and $\varpi$ be a uniformizer of $\cO$. Set $\kappa$ to denote the residue field $\cO/(\varpi)$. If $\rho$ is a representation of $\op{G}_\Q$, denote by $\rho_{|\ell}$ the restriction of $\rho$ to $\op{G}_p$. Then $\rho$ is \emph{unramified at $\ell$} if $\rho_{|\op{I}_\ell}$ is trivial. In this case, $\rho(\sigma_\ell)$ is well defined.
    \item Given a finite set of primes $S$, let $\Q_S\subset \bar{\Q}$ be the maximal extension of $\Q$ in which all primes $\ell\notin S$ are unramified. Set $\op{G}_S:=\op{Gal}(\Q_S/\Q)$ and $H^i(\Q_S/\Q, \cdot):=H^i(\op{G}_S, \cdot)$.
\end{itemize}

\subsection{Selmer and dual Selmer data}
\par Let \( S \) be a finite set of rational primes which includes \( p \) and suppose \( M \) is a finite-dimensional vector space over \( \mathbb{F}_p \), equipped with a continuous action of \( \operatorname{G}_S\). For \( i = 1,2 \), define
\[
\Sh_S^i(M) := \ker\left( H^i(\mathbb{Q}_S/\mathbb{Q}, M) \longrightarrow \bigoplus_{\ell \in S} H^i(\mathbb{Q}_\ell, M) \right).
\]
Now, for each prime \( \ell \in S \), fix a choice of subspace \( \mathcal{L}_\ell \subseteq H^1(\mathbb{Q}_\ell, M) \). The collection \( \mathcal{L} = \{ \mathcal{L}_\ell \}_{\ell \in S} \) specifies a local condition at each prime. Using these local conditions, one defines the \emph{Selmer group} \( H^1_{\mathcal{L}}(\mathbb{Q}_S/\mathbb{Q}, M) \) as follows:
\[
H^1_{\mathcal{L}}(\mathbb{Q}_S/\mathbb{Q}, M) := \ker\left( H^1(\mathbb{Q}_S/\mathbb{Q}, M) \longrightarrow \bigoplus_{\ell \in S} \frac{H^1(\mathbb{Q}_\ell, M)}{\mathcal{L}_\ell} \right).
\]
One introduces the \emph{dual module} \( M^* := \operatorname{Hom}(M, \mu_p) \), where \( \mu_p \) denotes the group of \( p \)-th roots of unity (viewed as a Galois module). 

With respect to the Tate pairing, for each \( \ell \in S \), the \emph{orthogonal complement} \( \mathcal{L}_\ell^\perp \subseteq H^1(\mathbb{Q}_\ell, M^*) \) of \( \mathcal{L}_\ell \) is defined to consist of all elements of \( H^1(\mathbb{Q}_\ell, M^*) \) that pair trivially with every element of \( \mathcal{L}_\ell \). The \emph{dual Selmer group} \( H^1_{\mathcal{L}^\perp}(\mathbb{Q}_S/\mathbb{Q}, M^*) \) is defined to be the subgroup of \( H^1(\mathbb{Q}_S/\mathbb{Q}, M^*) \) consisting of classes whose local restrictions lie inside the subspaces \( \mathcal{L}_\ell^\perp \) for each \( \ell \in S \), that is,
\begin{equation}\label{dselmergp}
H^1_{\mathcal{L}^\perp}(\mathbb{Q}_S/\mathbb{Q}, M^*) := \ker\left( H^1(\mathbb{Q}_S/\mathbb{Q}, M^*) \longrightarrow \bigoplus_{\ell \in S} \frac{H^1(\mathbb{Q}_\ell, M^*)}{\mathcal{L}_\ell^\perp} \right).
\end{equation}

\subsection{Modular Galois representations and level raising}
\par Let $f$ be a Hecke eigencuspform and denote the integral representation by $\rho_{T_f}:\op{G}_\Q\rightarrow \op{GL}_2(\cO)$. Let $\overline{\rho_{T_f}}:=\rho_{T_f}\pmod{(\varpi)}$ be the mod-$(\varpi)$ reduction of $\rho_{T_f}$. 

\begin{ass}
Assume throughout this article:
\begin{enumerate}
\item[(i)] $p \nmid N_f$,
\item[(ii)] the residual representation $\overline{\rho_{T_f}}$ is absolutely irreducible.
\item[(iii)] The modular form $f$ has weight $2$ with trivial nebentype. 
\item[(iv)] Assume that $f$ is $p$-ordinary, which is to say that $\iota_p(a_p(f))$ is a $p$-adic unit.
\end{enumerate}
\end{ass}
Condition (i) implies that the local representation $\rho_{|p}$ is crystalline. Condition (ii) ensures that $T_f$ is, up to scaling, the unique Galois-stable lattice in $V_f$. In light of this, and with no risk of confusion, we shall henceforth write $\rho_f$ (respectively, $\overline{\rho_f}$) in place of $\rho_{T_f}$ (respectively, $\overline{\rho_{T_f}}$), thereby omitting the explicit reference to the choice of lattice. Condition (iii) implies that $\op{det}\rho_f=\chi$. Condition (iv) implies that $\rho_{| p}=\mtx{\chi\gamma}{\ast}{0}{\gamma^{-1}}$ where $\gamma:\op{G}_p\rightarrow \Z_p^\times$ is a finite order unramified character.

\par Let $c \in \operatorname{G}_\Q$ denote complex conjugation. A Galois character $\eta$ is said to be \emph{odd} if $\eta(c) = -1$, and a representation $\rho$ is \emph{odd} if $\det\rho(c) = -1$. A continuous representation  
\[
\bar{\rho} : \mathrm{Gal}(\overline{\Q}/\Q) \to \mathrm{GL}_2(\bar{\F}_p)
\]
is said to be \emph{modular} if it arises as the residual representation associated to a Hecke eigencuspform $f$, i.e., if $\bar{\rho} \simeq \bar{\rho}_f$ for some choice of embedding $F_f \hookrightarrow \bar{\Q}_p$.

\par The weak form of Serre’s conjecture \cite{serre1987representations} asserts that every odd, irreducible, and continuous representation $\bar{\rho}$ as above is modular. Let $N_{\bar{\rho}}$ denote the prime-to-$p$ part of the Artin conductor of $\bar{\rho}$. Note that $N_{\bar{\rho}}$ is also called the \emph{Serre conductor} of $\bar{\rho}$. A modular form $f$ is said to be of \emph{optimal level} for $\bar{\rho}$ if it is new of level $\Gamma_1(N_{\bar{\rho}})$ and satisfies $\bar{\rho}_f \simeq \bar{\rho}$. The strong form of the conjecture asserts that every odd, irreducible, and continuous representation $\bar{\rho}$ arises from a cuspidal newform $g$ of weight $k \geq 2$ and level $\Gamma_1(N_{\bar{\rho}})$, where the weight $k = k(\bar{\rho})$ is explicitly prescribed in \cite[§2]{serre1987representations}; see also \cite[p.11]{ribet1999lectures}. This conjecture was ultimately proven by Khare and Wintenberger in \cite{khare2009serre, khare2009serre2}, building on foundational results of Ribet, notably his level-lowering theorem \cite{ribetlevellowering}.

Assume that  $\rho_f$ lifts $\bar{\rho}$. Then the optimal level $N_{\bar{\rho}}$ necessarily divides $N_f$. Carayol \cite{carayol1989representations} proved that there are explicit conditions on integers $N_f$ that arise as levels modular forms $f$.

\begin{theorem}[Carayol]\label{carayol thm}
The level $N_f$ admits a factorization
\[
N_f = N_{\bar{\rho}} \prod_\ell \ell^{\alpha(\ell)},
\]
and for each prime $\ell$ with $\alpha(\ell) > 0$, one of the following conditions holds:
\begin{enumerate}
    \item $\ell \nmid N_{\bar{\rho}}$, $\ell(\mathrm{tr} \, \bar{\rho}(\sigma_\ell))^2 = (1+\ell)^2 \det \bar{\rho}(\sigma_\ell)$ in $\bar{\F}_p$, and $\alpha(\ell) = 1$;
    \item $\ell \equiv -1 \pmod{p}$, and one of the following:
    \begin{enumerate}
        \item $\ell \nmid N_{\bar{\rho}}$, $\mathrm{tr} \, \bar{\rho}(\sigma_\ell) \equiv 0$ in $\bar{\F}_p$, and $\alpha(\ell) = 2$;
        \item $\ell \mid N_{\bar{\rho}}$, $\det \bar{\rho}$ is unramified at $\ell$, and $\alpha(\ell) = 1$;
    \end{enumerate}
    \item $\ell \equiv 1 \pmod{p}$, and one of the following:
    \begin{enumerate}
        \item $\ell \nmid N_{\bar{\rho}}$ and $\alpha(\ell) = 2$;
        \item either $\ell \mid N_{\bar{\rho}}$ and $\alpha(\ell) = 1$, or $\ell \nmid N_{\bar{\rho}}$ and $\alpha(\ell) = 1$.
    \end{enumerate}
\end{enumerate}
\end{theorem}

\begin{definition}
The set of levels $N_f$ satisfying the conditions in Theorem \ref{carayol thm} is denoted by $\mathcal{S}(\bar{\rho})$. For $X>0$, set $n(\bar{\rho}; X):=\#\{N_f\in \mathcal{S}(\bar{\rho})\mid N_f\leq X\}$.
\end{definition}
Diamond and Taylor proved a striking \emph{level raising} theorem for mod \( p \) Galois representations, which provides a converse to Theorem~\ref{carayol thm}.

\begin{theorem}[Diamond–Taylor \cite{diamond1994non}]\label{DT theorem}
Let $p \geq 5$ be a prime, and let
\[
\bar{\rho} : \mathrm{Gal}(\overline{\Q}/\Q) \to \mathrm{GL}_2(\bar{\F}_p)
\]
be an irreducible Galois representation. Suppose $\bar{\rho}$ arises from a newform $g$ on $\Gamma_1(N_{\bar{\rho}})$ of weight $k$ with $2 \leq k \leq p-2$, and let $M \in \mathcal{S}(\bar{\rho})$. Then there exists a cuspidal newform $f$ of weight $k$ and level $M$ such that $\bar{\rho}_f \simeq \bar{\rho}$.
\end{theorem}
Carayol's theorem gives necessary local conditions on the level, Diamond and Taylor’s theorem shows that these conditions are also sufficient, thereby completing the characterization of levels at which \( \bar{\rho} \) arises.

\subsection{Selmer groups}
\par  Associated to a Hecke eigenform \( f \), there are two natural Selmer groups attached to its \( p \)-adic Galois representation. The first is defined using Greenberg’s local condition at \( p \), together with the usual unramified conditions away from \( p \). This Selmer group is defined over $\Q$, and will differ from its Iwasawa theoretic counterpart, which is defined over the cyclotomic $\Z_p$-extension of $\Q$. We will refer to this as the \emph{\( \mathbb{Q} \)-Greenberg Selmer group} of \( f \). The second is the \emph{Bloch--Kato Selmer group}, which is defined via the framework of \( p \)-adic Hodge theory.

\par Since $f$ is $p$-ordinary, $V_f$ sits in a short exact sequence of $\cK[\op{G}_p]$-modules, 
\[0\rightarrow V_f^+\xrightarrow{\iota} V_f \xrightarrow{\pi} V_f^-\rightarrow 0, \] where $V_f^\pm$ are one-dimensional as $\cK$ vector spaces. The action of $\op{I}_p$ on $V_f^+$ is via $\chi_p^{k-1}$, where $k$ is the weight of $f$. On the other hand, the $\op{G}_p$ action on the quotient $V_f^-$ is unramified. Setting 
\[T_f^+:=T_f\cap V_f^+\text{ and } T_f^-:=\pi(T_f)\]
one arrives at an exact sequence of $\cO[\op{G}_p]$-modules
\[0\rightarrow T_f^+\rightarrow T_f\rightarrow T_f^-\rightarrow 0.\]
Likewise, we set $A_f^{\pm}:=V_f^\pm/T_f^\pm$ and obtain the exact sequence
\[0\rightarrow A_f^+\rightarrow A_f \rightarrow A_f^-\rightarrow 0.\]
Given a rational prime $\ell\neq p$, define $H^1_{\op{Gr}}(\Q_\ell, A_f):=H^1_{\op{nr}}(\Q_\ell, A_f)$. Define 
\[H^1_{\op{Gr}}(\Q_p, A_f):=\op{ker}\left\{ H^1(\Q_p, A_f)\rightarrow H^1(\op{I}_p, A_f^-)\right\}.\]
\begin{definition}[Greenberg Selmer group over $\Q$]
    With respect to notation above, the Greenberg Selmer group associated to the $p$-divisible module $A_f$ is defined as follows:
    \[\op{Sel}_{\op{Gr}}(A_f/\Q):=\op{ker}\left\{ H^1(\Q, A_f)\rightarrow \bigoplus_\ell \frac{H^1(\Q_\ell, A_f)}{H^1_{\op{Gr}}(\Q_\ell, A_f)} \right\}.\]
\end{definition}
\begin{remark}
Another natural Selmer group to consider was introduced by Bloch and Kato \cite{bloch1990functions}:
\[
    \Sel_{\op{BK}}(A_f/\Q)
    := \ker\!\left\{
        H^1(\Q, A_f)
        \longrightarrow
        \prod_{\ell}
        \frac{H^1(\Q_\ell, A_f)}{H^1_f(\Q_\ell, A_f)}
    \right\},
\]
where the local conditions $H^1_f(\Q_\ell, A_f)$ differ from the conditions $H^1_{\op{Gr}}(\Q_\ell, A_f)$ used to define the Greenberg Selmer group except at finitely many primes. One has a natural inclusion
\[
    \Sel_{\op{BK}}(A_f/\Q) \subseteq \Sel_{\Gr}(A_f/\Q),
\]
and it is known that if $A_f$ is $p$-crystalline, then $\Sel_{\Gr}(A_f/\Q)$ is a finite-index subgroup of $\Sel_{\op{BK}}(A_f/\Q)$ (see \cite[Proposition~4.1(1)]{ochiai2000control}, \cite[(3.1)]{LongoVigni}, and \cite[Theorem~3]{flach}). Consequently, when the Greenberg Selmer group vanishes, so does the Bloch--Kato Selmer group.
\end{remark}
\subsection{The residual Selmer group}

Let $A_f[\varpi]$ denote the residual representation, which may also be identified with $T_f / \varpi T_f$. Define a Greenberg Selmer group associated to $A_f[\varpi]$ as follows. For $\ell \neq p$, set  
\[
H^1_{\operatorname{Gr}}(\Q_\ell, A_f[\varpi]) := H^1_{\operatorname{nr}}(\Q_\ell, A_f[\varpi]),
\]
and at $p$, define  
\[
H^1_{\operatorname{Gr}}(\Q_p, A_f[\varpi]) := \ker\left\{ H^1(\Q_p, A_f[\varpi]) \longrightarrow H^1(\operatorname{I}_p, A_f^-[\varpi]) \right\}.
\]
The inclusion $A_f[\varpi] \hookrightarrow A_f$ induces a map of Selmer groups  
\[
\alpha \colon \operatorname{Sel}_{\operatorname{Gr}}(A_f[\varpi]/\Q) \longrightarrow \operatorname{Sel}_{\operatorname{Gr}}(A_f/\Q)[\varpi],
\]
which fits into the commutative diagram  
\begin{equation}\label{BK exact diagram}
\begin{tikzcd}[column sep = small, row sep = large]
0 \arrow{r} & \operatorname{Sel}_{\operatorname{Gr}}(A_f[\varpi]/\Q) \arrow{r} \arrow{d}{\alpha} & H^1(\Q, A_f[\varpi]) \arrow{r}{\Phi} \arrow{d}{g} & \bigoplus_{\ell} \frac{H^1(\Q_\ell, A_f[\varpi])}{H^1_{\operatorname{Gr}}(\Q_\ell, A_f[\varpi])} \arrow{d}{h} \\
0 \arrow{r} & \operatorname{Sel}_{\operatorname{Gr}}(A_f/\Q)[\varpi] \arrow{r} & H^1(\Q, A_f)[\varpi] \arrow{r} & \bigoplus_{\ell} \frac{H^1(\Q_\ell, A_f)[\varpi]}{H^1_{\operatorname{Gr}}(\Q_\ell, A_f)[\varpi]}.
\end{tikzcd}
\end{equation}
\noindent The map $h$ decomposes into a direct sum $h=\bigoplus_\ell h_\ell$, where
\[h_\ell:\frac{H^1(\Q_\ell, A_f[\varpi])}{H^1_{\operatorname{Gr}}(\Q_\ell, A_f[\varpi])}\longrightarrow \frac{H^1(\Q_\ell, A_f)[\varpi]}{H^1_{\operatorname{Gr}}(\Q_\ell, A_f)[\varpi]}.\]
We set $\beta_\ell(A_f):=\dim_{\kappa}\op{ker}h_\ell$.

\begin{lemma}\label{beta local lemma}
    The following assertions hold:
    \begin{enumerate}
        \item Suppose $\ell\neq p$, then there is a natural injection\begin{equation}\label{natural inj}\op{ker}h_\ell\hookrightarrow \left(\frac{H^0(\op{I}_\ell, A_f)}{\varpi H^0(\op{I}_\ell, A_f)}\right)^{\op{G}_\ell/\op{I}_\ell}.\end{equation}
        \item Let $\ell\neq p$ and suppose that $\ell\nmid N_{\bar{\rho}}$ (equivalently, $\bar{\rho}$ is unramified at $\ell$). Then, $\beta_\ell(A_f)=0$.
        \item One has that $\beta_p(A_f)=0$.
    \end{enumerate}
\end{lemma}
\begin{proof}
 Let us begin by proving part (1). Consider the commutative diagram
\[
\begin{array}{ccc}
\frac{H^1(\Q_\ell, A_f[\varpi])}{H^1_{\operatorname{nr}}(\Q_\ell, A_f[\varpi])} & \xrightarrow{h_\ell} & \frac{H^1(\Q_\ell, A_f)[\varpi]}{H^1_{\operatorname{nr}}(\Q_\ell, A_f)[\varpi]} \\
\downarrow^{} &  & \downarrow^{} \\
H^1(\op{I}_\ell, A_f[\varpi])^{\op{G}_\ell/\op{I}_\ell} & \xrightarrow{g_\ell} & H^1(\op{I}_\ell, A_f)^{\op{G}_\ell/\op{I}_\ell}[\varpi].
\end{array}
\]
The vertical maps are injective, and therefore, $\op{ker}h_\ell$ injects into $\op{ker} g_\ell$. From the Kummer sequence, 
\[\op{ker}g_\ell\simeq \left(\frac{H^0(\op{I}_\ell, A_f)}{\varpi H^0(\op{I}_\ell, A_f)}\right)^{\op{G}_\ell/\op{I}_\ell},\] from which the result follows.

\par Next, we prove (2). Since the action of $\op{I}_\ell$ on $A_f$ is trivial, we find that $H^0(\op{I}_\ell, A_f)=A_f$. Since $A_f$ is $\varpi$-divisible, $A_f/\varpi A_f=0$ and thus part (2) follows from (1).

\par In order to prove part (3), it suffices to show that the map 
\[H^1(\op{I}_p, A_f^-[\varpi])\rightarrow H^1(\op{I}_p, A_f^-)\] is injective. By the Kummer sequence, its kernel is isomorphic to $\frac{H^0(\op{I}_p, A_f^-)}{\varpi H^0(\op{I}_p, A_f^-)}$. Since the action of $\op{I}_p$ on $A_f^-$ is trivial, and $A_f^-$ is $\varpi$-divisible, we deduce that $h_p$ is injective. Consequently $\beta_p(A_f)=0$, proving (3).\end{proof}

\begin{proposition}\label{lower bound propn}
With the notation as above, one has
\[
\dim_{\kappa} \operatorname{Sel}_{\operatorname{Gr}}(A_f/\mathbb{Q})[\varpi] \leq \dim_{\kappa} \operatorname{Sel}_{\operatorname{Gr}}(A_f[\varpi]/\mathbb{Q}) + \sum_{\ell} \beta_\ell(A_f).
\]
\end{proposition}

\begin{proof}
In \eqref{BK exact diagram}, let \( h' \) denote the restriction of \( h \) to the image of \( \Phi \). Applying the snake lemma yields an exact sequence
\[
0 \longrightarrow \ker \alpha \longrightarrow \ker g \longrightarrow \ker h' \longrightarrow \operatorname{coker} \alpha \longrightarrow \operatorname{coker} g.
\]
Since \( A_f[\varpi] \) is irreducible as a \( \operatorname{G}_{\mathbb{Q}} \)-module, it follows that \( H^0(\mathbb{Q}, A_f[\varpi]) = 0 \). Thus, the Kummer sequence shows that \( g \) is an isomorphism. In particular, \( \alpha \) is injective and \( \operatorname{coker} \alpha \) injects into \( \ker h \).

There is therefore a left exact sequence
\begin{equation}\label{selmer left exact}
0 \longrightarrow \operatorname{Sel}_{\operatorname{Gr}}(A_f[\varpi]/\mathbb{Q}) \longrightarrow \operatorname{Sel}_{\operatorname{Gr}}(A_f/\mathbb{Q})[\varpi] \longrightarrow \bigoplus_{\ell} \ker h_\ell,
\end{equation}
from which the claimed inequality follows.
\end{proof}
\par A similar strategy yields lower bounds for the dimension of $\operatorname{Sel}_{\operatorname{Gr}}(A_f/\mathbb{Q})[\varpi]$ over $\kappa$. Let $S$ denote the set of primes dividing $N_f p$. For each $\ell \in S$, define the local condition $\cL_\ell := H^1_{\operatorname{Gr}}(\mathbb{Q}_\ell, A_f[\varpi])$, so that the associated Selmer group is given by

$$
H^1_{\cL}(\mathbb{Q}_S/\mathbb{Q}, A_f[\varpi]) = \operatorname{Sel}_{\operatorname{Gr}}(A_f[\varpi]/\mathbb{Q}).
$$

One defines the dual Selmer group by setting

$$
\operatorname{Sel}_{\operatorname{Gr}, \perp}(A_f[\varpi]^*/\mathbb{Q}) := H^1_{\cL^\perp}(\mathbb{Q}_S/\mathbb{Q}, A_f[\varpi]^*),
$$

as in \eqref{dselmergp}. While the result that follows is not used in the proof of the main theorem, it is of independent interest.

\begin{lemma}
    With respect to the above notation, 
    \[\begin{split}&\dim_{\kappa}\operatorname{Sel}_{\operatorname{Gr}}(A_f/\mathbb{Q})[\varpi]\geq \sum_{\ell\in S} \beta_\ell(A_f)
    + \dim_{\kappa} H^1_{\op{Gr}}(\Q_p, A_f[\varpi]) \\ -&\dim_{\kappa} H^0(\Q_p, A_f[\varpi])- \dim_{\kappa} H^0(\Q_\infty,  A_f[\varpi])+\dim_{\kappa} \Sh_S^2(A_f[\varpi]).\end{split}\]
\end{lemma}
\begin{proof}
    From \eqref{BK exact diagram} and the arguments in the proof of Proposition \ref{lower bound propn}, there is a short exact sequence
    \[
0 \longrightarrow \operatorname{Sel}_{\operatorname{Gr}}(A_f[\varpi]/\mathbb{Q}) \longrightarrow \operatorname{Sel}_{\operatorname{Gr}}(A_f/\mathbb{Q})[\varpi] \longrightarrow  \ker h'\rightarrow 0.
\]
From the Poitou--Tate sequence (cf. \cite[p.~555, line 7]{tayloricosahedral}), we find that 
\[\dim_{\kappa} \op{ker} h\leq \dim_{\kappa} \op{ker} h'+ \dim_{\kappa} \op{Sel}_{Gr, \perp} (A_f[\varpi]^*/\Q)-\dim_{\kappa} \Sh_S^2(A_f[\varpi]).\]
Thus, one finds that 
\begin{equation}\label{lemma 2.9 e1}
\begin{split}& \dim_{\kappa}\operatorname{Sel}_{\operatorname{Gr}}(A_f/\mathbb{Q})[\varpi]=  \dim_{\kappa}\operatorname{Sel}_{\operatorname{Gr}}(A_f[\varpi]/\mathbb{Q})+\dim_{\kappa}\ker h'\\
\geq &  \dim_{\kappa}\operatorname{Sel}_{\operatorname{Gr}}(A_f[\varpi]/\mathbb{Q})-\dim_{\kappa} \op{Sel}_{Gr, \perp} (A_f[\varpi]^*/\Q)+\sum_{\ell\in S} \dim_{\kappa} \op{ker}h_\ell+\dim_{\kappa} \Sh_S^2(A_f[\varpi]).
\end{split}\end{equation}
It follows from Wiles' formula \cite[Theorem 8.7.9]{NSW} that 
\begin{equation}\label{lemma 2.9 e2}\begin{split} &\dim_{\kappa}\operatorname{Sel}_{\operatorname{Gr}}(A_f[\varpi]/\mathbb{Q})-\dim_{\kappa} \op{Sel}_{Gr, \perp} (A_f[\varpi]^*/\Q)\\ 
=& \dim_{\mathbb{F}_p} H^0(\mathbb{Q}, A_f[\varpi]) - \dim_{\mathbb{F}_p} H^0(\mathbb{Q}, A_f[\varpi]^*) + \sum_{\ell \in S\cup \{\infty\}} \left( \dim_{\mathbb{F}_p} \mathcal{L}_\ell - \dim_{\mathbb{F}_p} H^0(\mathbb{Q}_\ell, A_f[\varpi]) \right)\\
= & \dim_{\kappa} H^1_{\op{Gr}}(\Q_p, A_f[\varpi])-\dim_{\kappa} H^0(\Q_p, A_f[\varpi])- \dim_{\kappa} H^0(\Q_\infty,  A_f[\varpi]).
\end{split}\end{equation}
Combining \eqref{lemma 2.9 e1} and \eqref{lemma 2.9 e2}, we arrive at our result.
\end{proof}

\section{Main results}
\par In this section, we fix a prime $p\geq 5$, a $p$-adic field $\cK$ with valuation ring $\cO$, uniformizer $\varpi$ and residue field $\kappa$. We also fix a surjective and odd Galois representation
\[\bar{\rho}:\op{G}_\Q\rightarrow \op{GL}_2(\kappa)\] with $\op{det}\bar{\rho}=\bar{\chi}$.
As is well known, the results of Khare and Wintenberger imply that there exists a Hecke eigencuspform form $f$ of weight $2$, trivial nebentype and optimal level $N_f=N_{\bar{\rho}}$ such that $\bar{\rho}_f=\bar{\rho}$.
\par Let $\Omega_{\bar{\rho}}$ be the set of prime numbers $\ell$ such that 
\begin{enumerate}
    \item $\ell\nmid N_{\bar{\rho}} p$, 
    \item $\ell\not\equiv \pm 1\mod{p}$,  \item$\bar{\rho}_f(\sigma_\ell)=\mtx{-\ell}{0}{0}{-1}$ for a suitable choice of basis. 
\end{enumerate}

\begin{remark}
    It follows from an easy application of the Chebotarev density theorem that if $p\geq 5$ and $\bar{\rho}_f$ is surjective, then $\Omega_{\bar{\rho}}$ has positive density. In fact, according to \cite[Proposition 5.2]{Rayblms} its density is equal to    \begin{equation}\label{density of Omega_f}
        \mathfrak{d}(\Omega_{\bar{\rho}})=\frac{(p-3)}{(p-1)^2}.
    \end{equation}
\end{remark}


\begin{lemma}\label{vanishing invariants for primes in Omega_f}
Let $g$ be a $p$-ordinary modular form of weight $2$ and trivial nebentypus such that $\bar{\rho}_g \simeq \bar{\rho}_f\simeq \bar{\rho}$. Suppose that $\ell \mid N_g$ is a prime lying in $\Omega_{\bar{\rho}}$. Then
\[
\beta_\ell(A_f) = 0 \quad \text{and} \quad \beta_\ell(A_g) = 0.
\]
\end{lemma}

\begin{proof}
Since $\ell\nmid N_{\bar{\rho}}$, the vanishing $\beta_\ell(A_f) = 0$ follows from part (2) of Lemma~\ref{beta local lemma}. For the vanishing of $\beta_\ell(A_g)$, part (1) of the same lemma reduces the problem to showing that
\[
\left( \frac{H^0(\mathrm{I}_\ell, A_g)}{\varpi H^0(\mathrm{I}_\ell, A_g)} \right)^{\mathrm{G}_\ell/\mathrm{I}_\ell} = 0.
\]
It follows from \cite[Lemma 4.4]{Rayblms} that $\rho_g$ factors through the maximal pro-$p$ tame inertia quotient of $\mathrm{G}_\ell$, generated by $\sigma_\ell$ and $\tau_\ell$ with the relation $\label{sigma tau reln}\sigma_\ell \tau_\ell \sigma_\ell^{-1} = \tau_\ell^\ell$. Here, $\sigma_\ell$ is a lift of Frobenius and $\tau_\ell$ is a generator of the tame inertia. Further, there exist $x \equiv -1 \pmod{\varpi}$ and $y \equiv 0 \pmod{\varpi}$ such that
\[
\rho_g(\sigma_\ell) = \begin{pmatrix} \ell x & 0 \\ 0 & x \end{pmatrix} \quad \text{and} \quad \rho_g(\tau_\ell) = \begin{pmatrix} 1 & y \\ 0 & 1 \end{pmatrix}.
\] Since $\rho_g$ is ramified at $\ell$, one must have $y \neq 0$. Write $y = \varpi^n u$ with $u \in \mathcal{O}^\times$. It follows that
\[
A_g^{\mathrm{I}_\ell} = \left( \mathcal{K}/\mathcal{O} \right) \cdot e_1 \oplus \left( \varpi^{-n} \mathcal{O}/\mathcal{O} \right) \cdot e_2.
\]
Thus,
\[
\frac{H^0(\mathrm{I}_\ell, A_g)}{\varpi H^0(\mathrm{I}_\ell, A_g)} \simeq \kappa \cdot e_2,
\]
as a $\mathrm{G}_\ell$-module. The element $\sigma_\ell$ acts on $\kappa \cdot e_2$ via multiplication by $x \bmod \varpi$, which is $-1$ by assumption. Since $p \neq 2$, this action is nontrivial, and so the invariants under $\mathrm{G}_\ell/\mathrm{I}_\ell$ vanish:
\[
\left( \frac{H^0(\mathrm{I}_\ell, A_g)}{\varpi H^0(\mathrm{I}_\ell, A_g)} \right)^{\mathrm{G}_\ell/\mathrm{I}_\ell} = 0.
\]
Hence, $\beta_\ell(A_g) = 0$, as desired.
\end{proof}

\begin{lemma}\label{lemma for primes dividing artin cond}
Let $g$ be a $p$-ordinary modular form of weight $2$ such that $\bar{\rho}_g \simeq \bar{\rho}_f \simeq \bar{\rho}$. Assume that $p\nmid N_g$. Let $\ell$ be a prime dividing $N_{\bar{\rho}}$ such that $\ell \not\equiv \pm 1 \pmod{p}$.
Then
\[
\beta_\ell(A_f) = 0 \quad \text{and} \quad \beta_\ell(A_g) = 0.
\]
\end{lemma}

\begin{proof}
Let $h$ denote $f$ or $g$. Assume by way of contradiction that $\beta_\ell(A_h)>0$. By Part (1) of Lemma \ref{beta local lemma} we have $H^{0}\left(I_{\ell}, A_{h}\right) \neq 0$, hence $H^{0}\left(I_{\ell}, A_{h}[\varpi]\right) \neq 0$. Since $\det\bar{\rho}$ is the mod $p$ cyclotomic character and $\ell\neq p$, we find that $\op{det}(\bar{\rho}_{|\op{I}_\ell})=1$. Consequently, there exists a basis $(\bar{e_1}, \bar{e_2})$ with respect to which the restriction $\bar{\rho}|_{\mathrm{I}_\ell}$ takes the form
\[
\bar{\rho}|_{\mathrm{I}_\ell} = \begin{pmatrix} 1 & * \\ 0 & 1 \end{pmatrix},
\]
so that $\bar{\rho}$ factors through the maximal pro-$p$ tame inertia quotient of $\mathrm{G}_\ell$, generated by $\sigma_\ell$ and $\tau_\ell$ (as in the proof of Lemma \ref{vanishing invariants for primes in Omega_f}).
Since $\ell \mid N_{\bar{\rho}}$, the representation $\bar{\rho}$ is ramified at $\ell$. Therefore there exists $y\neq 0$ for which $\bar{\rho}(\tau_\ell) = \begin{pmatrix} 1 & y \\ 0 & 1 \end{pmatrix}.$ We also write $\bar{\rho}(\sigma_\ell) = \begin{pmatrix} a & b \\ c & d \end{pmatrix}$. From the relation $\sigma_\ell \tau_\ell \sigma_\ell^{-1}=\tau_\ell^\ell$, we find that
\[
\bar{\rho}(\sigma_\ell) \bar{\rho}(\tau_\ell) = \begin{pmatrix} a & ay + b \\ c & cy + d \end{pmatrix},
\]
and
\[
\bar{\rho}(\tau_\ell)^\ell \bar{\rho}(\sigma_\ell) = \begin{pmatrix} 1 & \ell y \\ 0 & 1 \end{pmatrix} \begin{pmatrix} a & b \\ c & d \end{pmatrix} = \begin{pmatrix} a + c \ell y & b + d \ell y \\ c & d \end{pmatrix}.
\]
Equating these matrices, we obtain $c = 0$ and $a = d\ell$.
Hence,
\begin{equation}\label{bar rho ell}
\bar{\rho}(\sigma_\ell) = \begin{pmatrix} d\ell & b \\ 0 & d \end{pmatrix}.
\end{equation}
Note that $\det \bar{\rho}(\sigma_\ell)=\bar{\chi}(\sigma_\ell) = \ell$, and thus we conclude that $d^2 \ell = \ell$, and consequently, $d = \pm 1$. 
\par Since $\bar{\rho}\left(\tau_{\ell}\right)-1=\left(\begin{array}{ll}0 & y \\ 0 & 0\end{array}\right)$, we see that $\rho\left(\tau_{\ell}\right)$ has Smith Normal Form $\left(\begin{array}{cc}1 & 0 \\ 0 & w\end{array}\right)$ for some $w \in(\varpi)$. If $w=0$ then $A_{h}^{I_{\ell}} / \varpi A_{h}^{I_{\ell}}=0$, contradicting $\beta_\ell\left(A_{h}\right)>0$. Thus $w \neq 0$, giving $A_{h}^{I_{\ell}}=\left(\varpi^{-N} \mathcal{O} / \mathcal{O}\right) \cdot \gamma$ for some $N>0$ and some $\gamma$ which is necessarily congruent to $\bar{e}_{1}$ modulo $\varpi$. Since $\ell \neq \pm 1(\bmod \varpi)$, we see from (3.2) that the action of $\sigma_{\ell}-1$ on $A_{h}^{I_{\ell}} / \varpi A_{h}^{I_{\ell}}$ is invertible, hence $\left(A_{h}^{I_{\ell}} / \varpi A_{h}^{I_{\ell}}\right)^{G_{\ell} / I_{\ell}}=0$. Part (1) of Lemma \ref{beta local lemma} gives $\beta_\ell\left(A_{h}\right)=0$. This is a contradiction and the proof is complete. \end{proof}

\begin{proposition}\label{main propn}
Let $g$ be a $p$-ordinary modular form of weight $2$ such that $\bar{\rho}_g \simeq \bar{\rho}_f \simeq \bar{\rho}$. Assume the following:
\begin{enumerate}
    \item No prime $\ell \equiv \pm 1 \pmod{p}$ divides $N_{\bar{\rho}}$.
    \item Every prime dividing $N_g/N_{\bar{\rho}}$ lies in $\Omega_{\bar{\rho}}$.
\end{enumerate}
Then
\[
\dim_{\kappa} \mathrm{Sel}_{\mathrm{Gr}}(A_f/\Q)[\varpi] = \dim_{\kappa} \mathrm{Sel}_{\mathrm{Gr}}(A_g/\Q)[\varpi].
\]
\end{proposition}

\begin{proof}
Fix $h \in \{f, g\}$. We claim that
\[
\dim_{\kappa} \mathrm{Sel}_{\mathrm{Gr}}(A_h/\Q)[\varpi] = \dim_{\kappa} \mathrm{Sel}_{\mathrm{Gr}}(A_h[\varpi]/\Q).
\]
\noindent By the exactness of the sequence
\begin{equation}\label{boring exact sequence}
0 \to \mathrm{Sel}_{\mathrm{Gr}}(A_h[\varpi]/\Q) \to \mathrm{Sel}_{\mathrm{Gr}}(A_h/\Q)[\varpi] \to \bigoplus_\ell \op{ker}h_\ell,
\end{equation}
we have the inequality
\[
\dim_{\kappa} \mathrm{Sel}_{\mathrm{Gr}}(A_h[\varpi]/\Q) \leq \dim_{\kappa} \mathrm{Sel}_{\mathrm{Gr}}(A_h/\Q)[\varpi],
\]
and further,
\[
\dim_{\kappa} \mathrm{Sel}_{\mathrm{Gr}}(A_h/\Q)[\varpi] \leq \dim_{\kappa} \mathrm{Sel}_{\mathrm{Gr}}(A_h[\varpi]/\Q) + \sum_\ell \beta_\ell(A_h).
\]
It therefore suffices to show that $\beta_\ell(A_h) = 0$ for all primes $\ell$.

By Lemma~\ref{beta local lemma}, we have $\beta_\ell(A_h) = 0$ for $\ell = p$ and for all $\ell \nmid N_h$. For the remaining primes $\ell \mid N_h$, there are two cases:

\begin{itemize}
    \item If $\ell \mid (N_h/N_{\bar{\rho}})$, then $\ell \in \Omega_{\bar{\rho}}$ by assumption. Lemma~\ref{vanishing invariants for primes in Omega_f} then implies $\beta_\ell(A_h) = 0$.
    
    \item If $\ell \mid N_{\bar{\rho}}$, then $\ell \not\equiv \pm 1 \pmod{p}$ by assumption. Lemma~\ref{lemma for primes dividing artin cond} applies, and again we conclude $\beta_\ell(A_h) = 0$.
\end{itemize}

Hence, we obtain the desired equality:
\begin{equation}\label{dim equality}\dim_{\kappa} \mathrm{Sel}_{\mathrm{Gr}}(A_h/\Q)[\varpi] = \dim_{\kappa} \mathrm{Sel}_{\mathrm{Gr}}(A_h[\varpi]/\Q).
\end{equation}
The inertia group at $p$ acts on $A_h[\varpi]$ via $\mtx{\bar{\chi}}{\ast}{0}{1}$ and $A_h[\varpi]^+$ is the $1$-dimensional submodule on which $\op{I}_p$ acts via the mod-$p$ cyclotomic character and $A_h[\varpi]^-$ is the unique quotient on which $\op{I}_p$ acts trivially. Thus $A_h[\varpi]$ determines the $1$-dimensional modules $A_h[\varpi]^\pm$. 
Since $\bar{\rho}_f \simeq \bar{\rho}_g$, we deduce that
\[\dim_{\kappa} \mathrm{Sel}_{\mathrm{Gr}}(A_f[\varpi]/\Q)=\dim_{\kappa} \mathrm{Sel}_{\mathrm{Gr}}(A_g[\varpi]/\Q).\]
From \eqref{dim equality} it follows that 
\[\dim_{\kappa} \mathrm{Sel}_{\mathrm{Gr}}(A_f/\Q)[\varpi]=\dim_{\kappa} \mathrm{Sel}_{\mathrm{Gr}}(A_g/\Q)[\varpi].\]
\noindent This completes the proof.
\end{proof} Before proving our main theorem, we recall the following well known result of Serre.

\begin{theorem}[Serre]\label{Serre thm}
Let \( \Omega \) be a set of primes with Dirichlet density \( \delta > 0 \). For \( Y > 0 \), let \( M_\Omega(Y) \) denote the number of squarefree natural numbers \( M \leq Y \) divisible only by primes in \( \Omega \). Then
\[
M_\Omega(Y) \gg Y (\log Y)^{\delta - 1}.
\]
\end{theorem}

\begin{proof}
This follows from \cite[Theorem 2.4]{serredivisibilite}.
\end{proof}

\begin{proof}[Proof of Theorem \ref{main thm of intro}]
Let \( \Omega := \Omega_{\bar{\rho}} \), and consider any \( M \in M_\Omega(Y) \). Set \( N := M N_{\bar{\rho}} \). Then \( N \in \cS(\bar{\rho}) \), and by Theorem~\ref{DT theorem}, there exists a Hecke eigencuspform \( g \) such that \( \bar{\rho}_g \simeq \bar{\rho}_f \) and \( N = N_g \). By Proposition~\ref{main propn}, we have
\[
\dim_\kappa \op{Sel}_{\op{Gr}}(A_f/\Q)[\varpi] = \dim_\kappa \op{Sel}_{\op{Gr}}(A_g/\Q)[\varpi].
\]
Thus, \( N \in \cS_f(X) \), and hence
\[
N_f(X) \geq M_{\Omega}(X / N_{\bar{\rho}}).
\]
Applying Theorem~\ref{Serre thm}, we obtain
\[
N_f(X) \gg X (\log X)^{\alpha - 1},
\]
where \( \alpha = \mathfrak{d}(\Omega) = \frac{p - 3}{(p - 1)^2} \) by \eqref{density of Omega_f}.
\end{proof}

\bibliographystyle{alpha}
\bibliography{references}

@article{ribet1999lectures,
  title={Lectures on Serre’s conjectures},
  author={Ribet, Kenneth A and Stein, William A},
  journal={Arithmetic algebraic geometry (Park City, UT, 1999)},
  volume={9},
  pages={143--232},
  year={1999}
}

@article{khare2009serre,
  title={Serre’s modularity conjecture (I)},
  author={Khare, Chandrashekhar and Wintenberger, Jean-Pierre},
  journal={Inventiones mathematicae},
  volume={178},
  number={3},
  pages={485},
  year={2009},
  publisher={Springer}
}

@article{carayol1989representations,
  title={Sur les repr{\'e}sentations {G}aloisiennes modulo l attach{\'e}es aux formes modulaires},
  author={Carayol, Henri},
  journal={Duke Math. J},
  volume={59},
  number={3},
  pages={785--801},
  year={1989}
}

@article {ribetlevellowering,
    AUTHOR = {Ribet, K. A.},
     TITLE = {On modular representations of {${\rm Gal}(\overline{\bf
              Q}/{\bf Q})$} arising from modular forms},
   JOURNAL = {Invent. Math.},
  FJOURNAL = {Inventiones Mathematicae},
    VOLUME = {100},
      YEAR = {1990},
    NUMBER = {2},
     PAGES = {431--476},
}

@article{khare2009serre2,
  title={Serre’s modularity conjecture (II)},
  author={Khare, Chandrashekhar and Wintenberger, Jean-Pierre},
  journal={Inventiones mathematicae},
  volume={178},
  number={3},
  pages={505},
  year={2009},
  publisher={Springer}
}

@article{serre1987representations,
  title={Sur les repr{\'e}sentations modulaires de degr{\'e} 2 de Gal (Q/Q)},
  author={Serre, Jean-Pierre},
  year={1987}
}

@article{diamond1994non,
  title={Non-optimal levels of mod l modular representations},
  author={Diamond, Fred and Taylor, Richard},
  journal={Inventiones mathematicae},
  volume={115},
  pages={435--462},
  year={1994},
  publisher={Springer}
}

@article{ochiai2000control,
  title={Control Theorem for {B}loch--{K}ato's Selmer Groups of p-Adic Representations},
  author={Ochiai, Tadashi},
  journal={Journal of Number Theory},
  volume={82},
  number={1},
  pages={69--90},
  year={2000},
  publisher={Elsevier}
}

@article {MRH10,
    AUTHOR = {Mazur, B. and Rubin, K.},
     TITLE = {Ranks of twists of elliptic curves and {H}ilbert's tenth
              problem},
   JOURNAL = {Invent. Math.},
  FJOURNAL = {Inventiones Mathematicae},
    VOLUME = {181},
      YEAR = {2010},
    NUMBER = {3},
     PAGES = {541--575},
}

@inproceedings {Goldfeld,
    AUTHOR = {Goldfeld, Dorian},
     TITLE = {Conjectures on elliptic curves over quadratic fields},
 BOOKTITLE = {Number theory, {C}arbondale 1979 ({P}roc. {S}outhern
              {I}llinois {C}onf., {S}outhern {I}llinois {U}niv.,
              {C}arbondale, {I}ll., 1979)},
    SERIES = {Lecture Notes in Math.},
    VOLUME = {751},
     PAGES = {108--118},
 PUBLISHER = {Springer, Berlin},
      YEAR = {1979},
}

@incollection{bloch1990functions,
  title={L-functions and {T}amagawa numbers of motives},
  author={Bloch, Spencer and Kato, Kazuya},
  booktitle={The Grothendieck Festschrift: A Collection of Articles Written in Honor of the 60th Birthday of Alexander Grothendieck},
  pages={333--400},
  year={1990},
  publisher={Birkh{\"a}user Boston Boston, MA}
}

@article{greenberg2000iwasawa,
  title={On the {I}wasawa invariants of elliptic curves},
  author={Greenberg, Ralph and Vatsal, Vinayak},
  journal={Inventiones mathematicae},
  volume={142},
  number={1},
  pages={17--63},
  year={2000},
  publisher={Springer}
}

@article {OnoCrelle,
    AUTHOR = {Ono, Ken},
     TITLE = {Nonvanishing of quadratic twists of modular {$L$}-functions
              and applications to elliptic curves},
   JOURNAL = {J. Reine Angew. Math.},
  FJOURNAL = {Journal f\"{u}r die Reine und Angewandte Mathematik. [Crelle's
              Journal]},
    VOLUME = {533},
      YEAR = {2001},
     PAGES = {81--97},
}

@article {OnoSkinner,
    AUTHOR = {Ono, Ken and Skinner, Christopher},
     TITLE = {Non-vanishing of quadratic twists of modular {$L$}-functions},
   JOURNAL = {Invent. Math.},
  FJOURNAL = {Inventiones Mathematicae},
    VOLUME = {134},
      YEAR = {1998},
    NUMBER = {3},
     PAGES = {651--660},
}

@book {NSW,
    AUTHOR = {Neukirch, J\"{u}rgen and Schmidt, Alexander and Wingberg, Kay},
     TITLE = {Cohomology of number fields},
    SERIES = {Grundlehren der mathematischen Wissenschaften [Fundamental
              Principles of Mathematical Sciences]},
    VOLUME = {323},
   EDITION = {Second},
 PUBLISHER = {Springer-Verlag, Berlin},
      YEAR = {2008},
     PAGES = {xvi+825},
}

@article {tayloricosahedral,
    AUTHOR = {Taylor, Richard},
     TITLE = {On icosahedral {A}rtin representations. {II}},
   JOURNAL = {Amer. J. Math.},
  FJOURNAL = {American Journal of Mathematics},
    VOLUME = {125},
      YEAR = {2003},
    NUMBER = {3},
     PAGES = {549--566},
}

@incollection {serredivisibilite,
    AUTHOR = {Serre, Jean-Pierre},
     TITLE = {Divisibilit\'{e} de certaines fonctions arithm\'{e}tiques},
 BOOKTITLE = {S\'{e}minaire {D}elange-{P}isot-{P}oitou (16e ann\'{e}e: 1974/75),
              {T}h\'{e}orie des nombres, {F}asc. 1},
     PAGES = {Exp. No. 20, 28},
 PUBLISHER = {Secr\'{e}tariat Math., Paris},
      YEAR = {1975},
}

@article{smith2022distribution1,
  title={The distribution of $\ell^{\infty}$-Selmer groups in degree $\ell$ twist families I},
  author={Smith, Alexander},
  journal={arXiv preprint arXiv:2207.05674},
  year={2022}
}

@misc{smith2022distribution2,
      title={The distribution of $\ell^{\infty}$-Selmer groups in degree $\ell$ twist families II}, 
      author={Alexander Smith},
      year={2023},
      eprint={2207.05143},
      archivePrefix={arXiv},
      primaryClass={math.NT},
      url={https://arxiv.org/abs/2207.05143}, 
}

@article {KMR,
    AUTHOR = {Klagsbrun, Zev and Mazur, Barry and Rubin, Karl},
     TITLE = {A {M}arkov model for {S}elmer ranks in families of twists},
   JOURNAL = {Compos. Math.},
  FJOURNAL = {Compositio Mathematica},
    VOLUME = {150},
      YEAR = {2014},
    NUMBER = {7},
     PAGES = {1077--1106},
}

@article {KrizLi,
    AUTHOR = {Kriz, Daniel and Li, Chao},
     TITLE = {Goldfeld's conjecture and congruences between {H}eegner
              points},
   JOURNAL = {Forum Math. Sigma},
  FJOURNAL = {Forum of Mathematics. Sigma},
    VOLUME = {7},
      YEAR = {2019},
     PAGES = {Paper No. e15, 80},
      ISSN = {2050-5094},
   MRCLASS = {11G05 (11G40)},
  MRNUMBER = {3954912},
MRREVIEWER = {Gergely\ Z\'{a}br\'{a}di},
       DOI = {10.1017/fms.2019.9},
       URL = {https://doi.org/10.1017/fms.2019.9},
}

@article {Rayblms,
    AUTHOR = {Ray, Anwesh},
     TITLE = {Constructing {G}alois representations with prescribed
              {I}wasawa {$\lambda$}-invariant},
   JOURNAL = {Bull. Lond. Math. Soc.},
  FJOURNAL = {Bulletin of the London Mathematical Society},
    VOLUME = {56},
      YEAR = {2024},
    NUMBER = {5},
     PAGES = {1624--1642},
}

@article {flach,
    AUTHOR = {Flach, Matthias},
     TITLE = {A generalisation of the {C}assels-{T}ate pairing},
   JOURNAL = {J. Reine Angew. Math.},
  FJOURNAL = {Journal f\"{u}r die Reine und Angewandte Mathematik. [Crelle's
              Journal]},
    VOLUME = {412},
      YEAR = {1990},
     PAGES = {113--127},
}

@article {LongoVigni,
    AUTHOR = {Longo, Matteo and Vigni, Stefano},
     TITLE = {On {B}loch-{K}ato {S}elmer groups and {I}wasawa theory of
              {$p$}-adic {G}alois representations},
   JOURNAL = {New York J. Math.},
  FJOURNAL = {New York Journal of Mathematics},
    VOLUME = {27},
      YEAR = {2021},
     PAGES = {437--467},
}
\end{document}